\documentclass[10pt]{amsart}

\usepackage[margin=1.5in,letterpaper]{geometry}
\usepackage{amsmath,amssymb,mathtools}
\usepackage{needspace,etoolbox}
\usepackage[hidelinks]{hyperref}
\hypersetup{
  pdftitle={A Local Characterization of Unique Equilibrium States},
  pdfauthor={C. Evans Hedges}
}
\makeatletter
\patchcmd{\@setauthors}{\centering\footnotesize}{\centering\normalsize}{}{
  \PackageError{local-attainability}{Could not set author font size}{}
}
\patchcmd{\@setauthors}{\MakeUppercase{\authors}}{\authors}{}{
  \PackageError{local-attainability}{Could not preserve author capitalization}{}
}
\makeatother
\numberwithin{equation}{section}

\newtheorem{mainthm}{Theorem}
\newtheorem{lemma}{Lemma}[section]

\newcommand{\R}{\mathbb R}

\newcommand{\Z}{\mathbb Z}
\newcommand{\Eq}{\operatorname{Eq}}
\newcommand{\CM}{\mathcal M}
\newcommand{\CT}{\mathcal T}
\newcommand{\CH}{\mathcal H}
\newcommand{\zeromeasure}{\mathbf{0}}

\title[Unique Equilibrium States]{A Local Characterization of Unique Equilibrium States}
\author{C. Evans Hedges}
\date{September 19, 2026}

\begin{document}

\begin{abstract}
We study when a prescribed ergodic measure is the unique equilibrium
state of a continuous potential. For continuous actions of countably
infinite discrete amenable groups on compact metric spaces with finite
topological entropy, we characterize this property by a local condition
on entropy differences. Fixing an ergodic measure $\mu$, we extend
$h(\nu)-h(\mu)$ to the cone generated by the differences
$\nu-\mu$, scaling the entropy difference by the same factor. We prove
that $\mu$ is uniquely attainable if and only if this entropy difference
function is weak\nobreakdash-$*$ upper semicontinuous at the zero measure.

Upper semicontinuity of entropy at $\mu$ alone does not suffice. In fact, we
construct a shift space over a compact alphabet for which entropy is
upper semicontinuous at every ergodic measure and continuous at a measure supported on a fixed point, but this measure
is not an equilibrium state of any continuous potential.
\end{abstract}

\maketitle

\section{Introduction}

This paper is concerned with finding a continuous potential which has
a prescribed ergodic measure as its unique equilibrium state. We give
a local condition on entropy differences which characterizes when
such a potential exists. We also show that upper semicontinuity of
entropy at the prescribed measure alone is not enough.

Let $G$ be a countably infinite discrete amenable group acting
continuously on a nonempty compact metric space $X$ by
$T=\{T_g\}_{g\in G}$, and suppose that the action has finite topological
entropy. We write $K=M_T(X)$ for the invariant Borel probability
measures, with the weak\nobreakdash-$*$ topology, and
$h:K\to[0,\infty)$ for the measure-theoretic entropy function.
The real-valued continuous functions on $X$ are called
\textbf{potentials}, and we denote their space by $C(X)$. The
variational principle~\cite{ollagnier-pinchon} gives the
\textbf{topological pressure} of $\phi\in C(X)$ as
\[
 P(\phi)=\sup_{\nu\in K}\left(h(\nu)+\int\phi\,d\nu\right).
\]
The measures attaining this supremum, if any exist, are the
\textbf{equilibrium states} of $\phi$. We denote their collection by
$\Eq(\phi)$. A measure
$\mu\in K$ is \textbf{attainable} if $\mu\in\Eq(\phi)$ for some
$\phi\in C(X)$, in which case we say that $\phi$ \textbf{realizes} $\mu$.
It is \textbf{uniquely attainable} if $\Eq(\phi)=\{\mu\}$ for some
$\phi\in C(X)$.

This inverse problem was studied by Israel and
Phelps~\cite{israel-phelps}. In particular,
Phelps~\cite[Theorem~1]{phelps} proved that if $h$ is upper
semicontinuous on $K$, then every ergodic measure is uniquely
attainable. We ask what local condition is needed for one prescribed
ergodic measure $\mu$.

Upper semicontinuity of entropy at $\mu$ is necessary. Indeed, if
$\mu\in\Eq(\phi)$, then
\begin{equation}\label{eq:equilibrium-bound}
 h(\nu)-h(\mu)\leq\int\phi\,d(\mu-\nu)
 \qquad(\nu\in K),
\end{equation}
and the right side tends to zero as $\nu\to\mu$. The question of
whether this condition is sufficient was raised
in~\cite[Section~1.1]{hedges-freezing}. Theorem~\ref{thm:example}
gives a negative answer: there is a measure at which
entropy is continuous but which no continuous potential realizes.

Note here that, with $\mu$ fixed, the right side
of~\eqref{eq:equilibrium-bound} is a continuous affine function of
$\nu$ which vanishes at $\mu$. If we write $\lambda=\nu-\mu$, it
becomes the continuous linear function
$\lambda\mapsto-\int\phi\,d\lambda$. Thus a realizing potential
gives a continuous linear upper bound for the entropy differences
$h(\nu)-h(\mu)$. We can multiply both sides of the inequality by any
nonnegative scalar and keep the same bound on the scaled differences.

This leads us to consider the cone generated by $K$ after translating
$\mu$ to the zero measure, which we denote by $\zeromeasure$. For a
fixed ergodic $\mu\in K$, we define
\begin{equation}\label{eq:radial-cone}
 \CT_\mu K=\{t(\nu-\mu):t\geq0,\ \nu\in K\}.
\end{equation}
This is the \textbf{radial cone} of $K$ at $\mu$. Its elements are
signed measures describing directions in which we can move from $\mu$
while staying in $K$: if $\lambda\in\CT_\mu K$, then
$\mu+s\lambda\in K$ for all sufficiently small $s>0$. We write
$\CM(X)=C(X)^*$ for the space of finite signed Borel measures and
give the cone its induced weak\nobreakdash-$*$ topology.

We now subtract $h(\mu)$ from entropy and extend the resulting
differences to the cone by scaling. We define the
\textbf{entropy difference function} at $\mu$ by
\begin{equation}\label{eq:entropy-difference}
 D_\mu h\bigl(t(\nu-\mu)\bigr)
 =t\bigl(h(\nu)-h(\mu)\bigr).
\end{equation}
For $t=1$, this is just the difference $h(\nu)-h(\mu)$. Affinity of
entropy ensures that the value does not depend on how we represent
the signed measure $t(\nu-\mu)$. The function is defined on
$\CT_\mu K$, rather than on all signed measures.

Our question is now whether this entropy difference function has a
continuous linear upper bound. We show that upper semicontinuity at
$\zeromeasure$ is exactly the condition needed.

\begin{mainthm}\label{thm:main}
Let a countably infinite discrete amenable group act continuously on a compact
metric space with finite topological entropy, and let $\mu\in K$ be
ergodic. Then $\mu$ is uniquely attainable if and only if its entropy
difference function $D_\mu h$ is
weak\nobreakdash-$*$ upper semicontinuous at $\zeromeasure$ on $\CT_\mu K$.
\end{mainthm}

To prove unique attainability, it is enough to find any potential
which realizes $\mu$. Indeed,
Jenkinson's ergodic optimization theorem~\cite[Theorem~1]{jenkinson}
gives $u\in C(X)$ such that $\int u\,d\nu<\int u\,d\mu$ for every
$\nu\in K\setminus\{\mu\}$. If $\phi$ realizes $\mu$, adding this
strict inequality to~\eqref{eq:equilibrium-bound} shows that $\phi+u$
uniquely realizes $\mu$.

The condition in Theorem~\ref{thm:main} only needs entropy near $\mu$.
We obtain the same cone and the same function $D_\mu h$ if we use
only measures from any neighborhood of $\mu$. To see why, we can move
any $\nu$ toward $\mu$ along the line segment joining them, then
rescale the resulting difference to recover $\nu-\mu$. Entropy is
affine, so its difference from $h(\mu)$ rescales in the same way.

There is still an extra requirement beyond upper semicontinuity of
$h$ at $\mu$. For each $\epsilon>0$, we need a neighborhood of $\zeromeasure$
in which $t(h(\nu)-h(\mu))<\epsilon$ whenever $t(\nu-\mu)$ belongs
to that neighborhood. This must hold for all $t\geq0$ and $\nu\in K$,
using the same neighborhood. The usual upper semicontinuity condition
only concerns the unscaled entropy differences.

The following example shows why the scaling matters.

\begin{mainthm}\label{thm:example}
There is a shift space over a compact alphabet with finite topological
entropy whose entropy map is upper semicontinuous at every ergodic
measure. It has a fixed point $a$ such that entropy is continuous at
$\delta_a$, but $\delta_a$ is not an equilibrium state of any
continuous potential.
\end{mainthm}

Thus upper semicontinuity at every ergodic measure cannot replace
the global upper semicontinuity assumption in Phelps's theorem.

The rest of the paper is organized as follows. In Section~\ref{sec:cone}
we prove Theorem~\ref{thm:main} using separation in the space of signed
measures. In Section~\ref{sec:example} we construct the example and
prove Theorem~\ref{thm:example}.

\section{The entropy difference function}\label{sec:cone}

We fix an ergodic measure $\mu\in K$ and write $e=h(\mu)$. Under our
assumptions, $K$ is compact and metrizable, and $h$ is bounded and
nonnegative. We will also use the ergodic decomposition formula for
entropy and, in particular, its affinity under finite convex
combinations; see~\cite{downarowicz,kerr-li}. To shorten the formulas,
we write $\lambda(f)=\int f\,d\lambda$ for a signed measure $\lambda$
and a function $f\in C(X)$.

We first check that the entropy difference function is well-defined.
Suppose the same signed measure is written in two ways,
$t(\nu-\mu)=s(\eta-\mu)$ with $t\geq s>0$. Then
\[
 \nu=\frac{s}{t}\eta+\left(1-\frac{s}{t}\right)\mu,
\]
so affinity of entropy gives $t(h(\nu)-e)=s(h(\eta)-e)$. The two
representations therefore give the same value. If one of the scalars
is zero, both sides are zero. In particular, $D_\mu h(\zeromeasure)=0$.

We next check how this function behaves under addition and scaling.
For $t,s\geq0$ with $t+s>0$, we have
\[
 t(\nu-\mu)+s(\eta-\mu)
 =(t+s)\left(\frac{t\nu+s\eta}{t+s}-\mu\right).
\]
The measure $(t\nu+s\eta)/(t+s)$ belongs to $K$, so the sum again
belongs to $\CT_\mu K$. Affinity of entropy shows that the value of
$D_\mu h$ at this sum is the sum of its values at the two terms.
The definition also gives $D_\mu h(r\lambda)=rD_\mu h(\lambda)$ for
$r\geq0$. Thus $\CT_\mu K$ is a convex cone, and the entropy
difference function respects addition and nonnegative scaling.

We next verify that we can construct this cone using only measures
near $\mu$.

\begin{lemma}\label{lem:locality}
If $U$ is any weak\nobreakdash-$*$ neighborhood of $\mu$ in $K$, then
\[
 \CT_\mu K=\{t(\nu-\mu):t\geq0,\ \nu\in U\}.
\]
The entropy difference function defined using $U$ is the same function
$D_\mu h$.
\end{lemma}

\begin{proof}
For any $\nu\in K$, we can choose $s\in(0,1)$ small enough that
$\nu_s=(1-s)\mu+s\nu$ belongs to $U$. For every $t\geq0$, we then have
\[
 t(\nu-\mu)=\frac{t}{s}(\nu_s-\mu),
 \qquad
 t(h(\nu)-e)=\frac{t}{s}(h(\nu_s)-e).
\]
This proves both statements.
\end{proof}

\begin{proof}[Proof of Theorem~\ref{thm:main}]
Suppose first that a potential $\phi$ realizes $\mu$. Multiplying
\eqref{eq:equilibrium-bound} by any $t\geq0$ gives
\[
 D_\mu h(\lambda)\leq-\lambda(\phi)
 \qquad(\lambda\in\CT_\mu K).
\]
The right side is continuous and equals $0$ at $\zeromeasure$. Since
$D_\mu h(\zeromeasure)=0$ as well, this proves upper semicontinuity at
$\zeromeasure$.

Now suppose that $D_\mu h$ is upper semicontinuous at $\zeromeasure$. We want
to find a continuous function $\phi$ such that
$D_\mu h(\lambda)\leq-\lambda(\phi)$ for every $\lambda\in\CT_\mu K$.
Taking $\lambda=\nu-\mu$ would then give~\eqref{eq:equilibrium-bound},
so $\phi$ would realize $\mu$. To find this upper bound, we consider
the set of points on or below the graph of $D_\mu h$:
\begin{equation}\label{eq:hypograph-cone}
 \CH_\mu=
 \{(\lambda,r):\lambda\in\CT_\mu K,\ r\leq D_\mu h(\lambda)\}
 \subset\CM(X)\times\R.
\end{equation}
Since $D_\mu h(\zeromeasure)=0$, the set $\CH_\mu$ contains $(\zeromeasure,0)$. If
$(\lambda,r),(\eta,s)\in\CH_\mu$ and $a,b\geq0$, then
$a\lambda+b\eta\in\CT_\mu K$ and
\[
 ar+bs\leq aD_\mu h(\lambda)+bD_\mu h(\eta)
 =D_\mu h(a\lambda+b\eta).
\]
Thus $(a\lambda+b\eta,ar+bs)\in\CH_\mu$, so $\CH_\mu$ is a convex
cone. Its closure is also a convex cone.

We will separate this set from $(\zeromeasure,1)$. Since
$D_\mu h(\zeromeasure)=0$, upper semicontinuity at $\zeromeasure$ gives
an open weak\nobreakdash-$*$ neighborhood $V$ of $\zeromeasure$ such that
$D_\mu h(\lambda)<1/2$ for every $\lambda\in V\cap\CT_\mu K$.
The product of $V$ with the interval $(1/2,3/2)$ is therefore an open neighborhood of
$(\zeromeasure,1)$ disjoint from $\CH_\mu$. This proves that
$(\zeromeasure,1)\notin\overline{\CH_\mu}$.

We now apply the Hahn--Banach separation theorem in the form that
strictly separates a point from a nonempty closed convex set in a
locally convex space~\cite[Theorem~3.4(b)]{rudin}. We apply it to
$(\zeromeasure,1)$ and $\overline{\CH_\mu}$ in $\CM(X)\times\R$,
with the product of the weak\nobreakdash-$*$ topology and the usual
topology on $\R$. This gives a continuous linear function $L$ such that
\[
 L(\zeromeasure,1)>\sup_{(\lambda,r)\in\overline{\CH_\mu}}L(\lambda,r).
\]
For the weak\nobreakdash-$*$ topology on $\CM(X)$, we can write
\[
 L(\lambda,r)=\lambda(f)+cr
 \qquad(f\in C(X),\ c\in\R).
\]
For this choice of $L$, the function $f$ and the coefficient $c$ are
fixed. We now show that $c$ must be positive.

Suppose $L(\lambda,r)>0$ at some point of $\overline{\CH_\mu}$.
Since this set is a cone, $(t\lambda,tr)$ also belongs to it for
every $t\geq0$. By linearity,
$L(t\lambda,tr)=tL(\lambda,r)$, which tends to infinity as
$t\to\infty$. This contradicts the separation inequality, which
bounds all these values above by the fixed number $L(\zeromeasure,1)$.
Thus $L\leq0$ on $\overline{\CH_\mu}$. Since $(\zeromeasure,0)$ belongs to
this set and $L(\zeromeasure,0)=0$, the supremum is exactly zero. The separation
inequality now gives $c=L(\zeromeasure,1)>0$. We have proved that
\[
 \lambda(f)+cr\leq0\quad\text{on }\overline{\CH_\mu},
 \qquad c>0.
\]
For every $\nu\in K$, the point $(\nu-\mu,h(\nu)-e)$ belongs to
$\CH_\mu$. Substituting it into this inequality and dividing by $c$
gives
\[
 h(\nu)+\nu(f/c)\leq e+\mu(f/c)
 \qquad(\nu\in K).
\]
Thus $f/c$ is a realizing potential. As explained in the introduction,
we can add a function supplied by Jenkinson's theorem to make $\mu$
its unique equilibrium state.
\end{proof}

\section{Entropy continuity does not imply attainability}\label{sec:example}

We construct a shift space with two fixed points $a,b$ and ergodic
measures $\rho_{n,k}$ with entropy $(\log2)/n$ such that, for each
fixed $n\geq2$,
\[
 \rho_{n,k}\longrightarrow
 \left(1-\frac{1}{n^2}\right)\delta_a+\frac{1}{n^2}\delta_b
 \qquad(k\to\infty).
\]
As $n\to\infty$, the entropy tends to zero more slowly than the
coefficient of $\delta_b$. This will rule out any continuous
potential realizing $\delta_a$. We will also show that entropy is
continuous at $\delta_a$ and upper semicontinuous at every ergodic
measure. The construction adapts the compact-alphabet example
of Walters~\cite[Section~2]{walters}.

For $n\geq2$ and $k\geq1$, let $\epsilon_{n,k}=2^{-n}3^{-k}$ and
set
\[
 \mathcal A=
 \{(v,0,0):v\in\{0,1\}\}
 \ \cup\
 \{(v,\epsilon_{n,k},\epsilon_{n,k}c):
      v,c\in\{0,1\},\ n\geq2,\ k\geq1\}.
\]
The numbers $\epsilon_{n,k}$ are distinct and have zero as their only
accumulation point, so $\mathcal A$ is compact. We give
$\mathcal A^\Z$ the product topology and write $\sigma$ for the left
shift.

Let $w^{n,k}$ be the periodic binary sequence obtained by repeating
\[
 0^{(n^2-1)nk}1^{nk},
\]
with this word beginning at coordinate $0$. We define $Y_{n,k}$ to
be the set of all shifts of the sequences
\[
 \bigl(w_i^{n,k},\epsilon_{n,k},\epsilon_{n,k}c_i\bigr)_{i\in\Z},
 \qquad c_i\in\{0,1\},\quad c_i=0\text{ if }i\notin n\Z,
\]
and set
\[
 X=\overline{\bigcup_{n\geq2,\ k\geq1}Y_{n,k}}.
\]
The period $n^3k$ is divisible by $n$, so each $Y_{n,k}$ is a finite
union of closed sets and is shift invariant. Thus $X$ is a compact
shift space. Let $a,b$ be the constant sequences with symbols
$(0,0,0)$ and $(1,0,0)$, respectively. Both belong to $X$: fix $n$
and let $k\to\infty$, shifting the origin into the middle of the
corresponding constant run. In our first lemma we characterize the invariant measures on $X$: 

\begin{lemma}\label{lem:example-measures}
Every ergodic invariant probability on $X$ is supported on some
$Y_{n,k}$ or is one of $\delta_a,\delta_b$.
\end{lemma}

\begin{proof} Each $Y_{n,k}$ is closed and invariant and its second coordinate $\epsilon_{n,k}$
is isolated, so taking the closure adds only points with second (and therefore third)
coordinate zero. By construction the shorter
run in $w^{n,k}$ has length $nk$, which tends to infinity as
$n+k\to\infty$. Thus the first coordinate of any added point can
change at most once.

The nonconstant added points lie in two infinite shift orbits.
Invariance gives equal mass to every point in an orbit, so any
invariant probability gives these countable orbits zero mass.
Since the sets $Y_{n,k}$ and the two fixed points are disjoint and
invariant, any ergodic measure must be supported on one of them.
\end{proof}

We write $v(x)$ for the first coordinate of $x_0$.

\begin{lemma}\label{lem:components}
The topological entropy of $Y_{n,k}$ is $(\log2)/n$, and every
invariant probability $\nu$ on $Y_{n,k}$ satisfies $\nu(v)=1/n^2$.
Further, there are ergodic measures $\rho_{n,k}$ of maximal entropy on
$Y_{n,k}$ such that, for each fixed $n$,
\begin{equation}\label{eq:component-limit}
 \rho_{n,k}\longrightarrow
 \left(1-\frac{1}{n^2}\right)\delta_a+\frac{1}{n^2}\delta_b
 \qquad(k\to\infty).
\end{equation}
\end{lemma}

\begin{proof}
In $Y_{n,k}$, the first coordinate is periodic, the second is fixed,
and the third has one free binary choice every $n$ positions.
The periodic component does not affect the exponential growth
rate of words, so the topological entropy is $(\log2)/n$.
For every point in $Y_{n,k}$ the first coordinate has proportion $1/n^2$ of $1$s in each period;
averaging over a period gives $\nu(v)=1/n^2$ for every invariant $\nu$.

Each $Y_{n,k}$ is a subshift over a finite alphabet, so we may choose
an ergodic measure of maximal entropy $\rho_{n,k}$ on it.
Fix $n$ and let $k\to\infty$. Since $\epsilon_{n,k}\to0$, every
weak\nobreakdash-$*$ limit of these measures is supported on points
with second coordinate zero. By Lemma~\ref{lem:example-measures}
and ergodic decomposition, such a limit is a convex combination of
$\delta_a$ and $\delta_b$. Continuity of $v$ and
$\rho_{n,k}(v)=1/n^2$ determine the coefficient of $\delta_b$ to be
$1/n^2$. Thus every limit point is the measure
in~\eqref{eq:component-limit}, proving convergence.
\end{proof}

By Lemma~\ref{lem:example-measures}, ergodic decomposition and the
variational principle now give
\[
 h_{\mathrm{top}}(X,\sigma)
 =\sup_{n\geq2,\ k\geq1}h_{\mathrm{top}}(Y_{n,k},\sigma)
 =\frac{\log2}{2}<\infty.
\]

\begin{lemma}\label{lem:example-continuity}
The entropy map is upper semicontinuous at every ergodic measure.
It is continuous at $\delta_a$ and $\delta_b$.
\end{lemma}

\begin{proof}
We first consider an ergodic measure $\mu$ supported on
$Y=Y_{n,k}$. The second coordinate $\epsilon_{n,k}$ is isolated,
so $Y$ is both open and closed in $X$. If invariant probabilities
$\nu_j$ converge to $\mu$, then $p_j=\nu_j(Y)\to1$. For all
sufficiently large $j$, the normalized restrictions
$\eta_j=\nu_j|_Y/p_j$ are invariant probabilities on $Y$ and converge
to $\mu$. Entropy is upper semicontinuous on the finite-alphabet
subshift $Y$. Since $Y$ and its complement are invariant, affinity
and the entropy bound on $X$ give
\[
 h(\nu_j)\leq p_j h(\eta_j)+(1-p_j)\frac{\log2}{2}.
\]
Taking the upper limit gives $\limsup_j h(\nu_j)\leq h(\mu)$, as
required.

By Lemma~\ref{lem:components}, every invariant probability $\eta$
on a component $Y_{n,k}$ satisfies
\[
 h(\eta)\leq(\log2)\sqrt{\eta(v)}.
\]
This inequality also holds at the two fixed-point measures.
Ergodic decomposition and concavity of the square root therefore
give the same inequality for every invariant probability.
Since $v$ is continuous and $\delta_a(v)=0$, it follows that
$h(\nu)\to0$ as $\nu\to\delta_a$.

Finally, suppose $\nu_j\to\delta_b$. Every invariant probability
on a component $Y_{n,k}$ has $v$-average at most $1/4$, whereas
$\nu_j(v)\to1$. The total mass on these components therefore tends
to zero. Their entropies are uniformly bounded and the fixed-point
measures have entropy zero, so ergodic decomposition gives
$h(\nu_j)\to0$ as well.
\end{proof}

\begin{lemma}\label{lem:example-unattainable}
The measure $\delta_a$ is not an equilibrium state of any continuous
potential.
\end{lemma}

\begin{proof}
If $\delta_a$ were an equilibrium state of $\phi\in C(X)$, comparison
with $\rho_{n,k}$ would give
\[
 \frac{\log2}{n}+\rho_{n,k}(\phi)\leq\phi(a).
\]
For fixed $n$, letting $k\to\infty$ in~\eqref{eq:component-limit}
and rearranging gives
\[
 n\log2\leq\phi(a)-\phi(b).
\]
The right side is fixed, whereas the left side tends to infinity
with $n$, a contradiction.
\end{proof}

Combining these results proves Theorem~\ref{thm:example}.

\Needspace{16\baselineskip}
\bibliographystyle{plain}
\bibliography{references}

\end{document}